\documentclass[12pt]{amsart}
\usepackage[colorlinks=true,pagebackref,hyperindex,citecolor=magenta,linkcolor=Blue]{hyperref}
\usepackage[top=1in, bottom=1in, left=1.1in, right=1.1in]{geometry}
\usepackage{amsmath}
\usepackage{amsfonts}
\usepackage{amssymb}
\usepackage{amsthm}
\usepackage{stmaryrd}
\usepackage{braket}
\usepackage[all]{xy}
\usepackage{mathrsfs}
\usepackage{enumitem} 
\usepackage[T1]{fontenc}
\usepackage{color}
\usepackage{tikz-cd}
\usepackage{tikz}
\usetikzlibrary{matrix,arrows}

\makeatletter
\def\namedlabel#1#2{\begingroup
#2%
\def\@currentlabel{#2}%
\phantomsection\label{#1}\endgroup
}
\makeatother

\theoremstyle{theorem} 
\newtheorem{theorem}{Theorem}[section]

\newtheorem{corollary}[theorem]{Corollary}

\newtheorem{lemma}[theorem]{Lemma}
\newtheorem{proposition}[theorem]{Proposition}

\newtheorem{theoremx}{Theorem}

\theoremstyle{definition} 
\newtheorem{definition}[theorem]{Definition}

\newtheorem{remark}[theorem]{Remark}
\numberwithin{equation}{subsection}

\newcommand{\cR}{\mathscr{R}}

\newcommand{\twopartdef}[4]
{
	\left\{
		\begin{array}{ll}
			#1 & \mbox{if } #2 \\
			#3 & \mbox{if } #4
		\end{array}
	\right.
}

\newcommand{\twopartdefo}[3]
{
	\left\{
		\begin{array}{ll}
			#1 & \mbox{if } #2 \\
			#3 & \mbox{otherwise.}
		\end{array}
	\right.
}

\newcommand{\threepartdef}[6]
{
	\left\{
		\begin{array}{lll}
			#1 & \mbox{if } #2 \\
			#3 & \mbox{if } #4 \\
			#5 & \mbox{if } #6
		\end{array}
	\right.
}

\makeatletter
\def\@tocline#1#2#3#4#5#6#7{\relax
  \ifnum #1>\c@tocdepth % then omit
  \else
    \par \addpenalty\@secpenalty\addvspace{#2}%
    \begingroup \hyphenpenalty\@M
    \@ifempty{#4}{%
      \@tempdima\csname r@tocindent\number#1\endcsname\relax
    }{%
      \@tempdima#4\relax
    }%
    \parindent\z@ \leftskip#3\relax \advance\leftskip\@tempdima\relax
    \rightskip\@pnumwidth plus4em \parfillskip-\@pnumwidth
    #5\leavevmode\hskip-\@tempdima
      \ifcase #1
       \or\or \hskip 1.9em \or \hskip 2em \else \hskip 3em \fi%
      #6\nobreak\relax
    \dotfill\hbox to\@pnumwidth{\@tocpagenum{#7}}\par
    \nobreak
    \endgroup
  \fi}
\makeatother

\newcommand{\Ass}{\operatorname{Ass}}

\newcommand{\gr}{\operatorname{gr}}

\newcommand{\Ht}{\operatorname{ht}}

\newcommand{\p}{\mathfrak{p}}
\newcommand{\q}{\mathfrak{q}}

\newcommand{\m}{\mathfrak{m}}

\definecolor{blue-violet}{rgb}{0.54, 0.17, 0.89}
\definecolor{Blue}{rgb}{0.01, 0.28, 1.0}
\definecolor{gGreen}{rgb}{0.2, 0.8, 0.2}
\definecolor{Green}{rgb}{0.04, 0.85, 0.32}

\begin{document}

\title[On the Symbolic F-splitness of Binomial Edge Ideals]{On the Symbolic F-splitness of Binomial Edge Ideals}

\author[P. Ram\'irez-Moreno]{Pedro Ram\'irez-Moreno{$^1$}}
\address{Departamento de Matem\'aticas, Centro de Investigaci\'on en Matem\'aticas, M\'exico}
\email{pedro.ramirez@cimat.mx}

\thanks{{$^1$}This author was partially supported by CONAHCyT Scholarship 771118 and by CONAHCyT Grant 284598.}


\maketitle

\begin{abstract}
We study the symbolic $F$-splitness of families of binomial edge ideals. We also study the strong $F$-regularity of the symbolic blowup algebras of families of binomial edge ideals. We make use of Fedder-like criteria and combinatorial properties of the graphs associated to the binomial edge ideals in order to approach the aforementioned scenarios.
\end{abstract}

\setcounter{tocdepth}{1}
\tableofcontents

\section{Introduction}

Symbolic powers of an ideal $I$ have seen utility in commutative algebra and algebraic geometry over the last few decades. In the case of polynomial rings over a field $k$, they can be characterized as the intersection of powers of maximal ideals containing $I$ \cite{EisenbudHochster}. Furthermore, if $k$ is algebraically closed, one can relate them to the vanishing of differential operators \cite{Zariski}. The drawback is that generators for symbolic powers are usually harder to compute than generators for ordinary powers. Moreover, it turns out that symbolic powers and ordinary powers do not always coincide.

For monomial edge ideals, it is known that their symbolic and ordinary powers coincide if and only if the graph associated to the ideal is bipartite \cite{MR1283294}. In the binomial edge ideal case it is yet unknown if such a characterization exists. Efforts on this line of research have yielded some partial results. Some examples of families of graphs such that the symbolic and ordinary powers of their associated binomial edge ideals coincide are closed graphs \cite{MR4143239}, complete multipartite graphs \cite{Ohtani} and caterpillar graphs \cite{MR4520291}.

Related to this question is the $F$-splitness of the associated graded ring of these families of ideals and the symbolic $F$-splitness of these families of ideals \cite[Proposition 5.5]{destefani2021blowup}. Symbolic $F$-splitness is a stronger kind of $F$-singularity: if an ideal $I$ is symbolic $F$-split then $I$ is $F$-split. In general, both notions are not the same \cite{destefani2021blowup}. 

On this paper we study the symbolic $F$-splitness of binomial edge ideals of complete multipartite graphs, caterpillar graphs, graphs whose binomial edge ideal has only $2$ associated primes, and traceable unmixed graphs. In order to do this we develop combinatorial criteria that allows us to determine if an ideal is symbolic $F$-split. Using the same ideas, we also study the strong $F$-regularity of symbolic Rees algebras of these families of binomial edge ideals.

For a graph $G$, we denote $J_G$ the binomial edge ideal associated to $G$ and we denote $\cR^s(J_G)$ the symbolic Rees algebra associated to $J_G$.
The main results of this paper are the following ones.

\begin{theoremx}
Let $G$ be a simple connected graph such that it is one of the following graphs:
\begin{itemize}
\item[$1)$] A complete multipartite graph.
\item[$2)$] A caterpillar graph.
\item[$3)$] A graph such that $|\Ass(J_G)| \leq 2$.
\item[$4)$] A traceable graph such that $J_G$ is unmixed.
\end{itemize}
Then $J_G$ is symbolic F-split.
\end{theoremx}

\begin{theoremx}
Let $G$ be a simple connected graph such that it is one of the following graphs:
\begin{itemize}
\item[$1$)] A complete multipartite graph.
\item[$2)$] A graph such that $|\Ass(J_G)| \leq 2$.
\item[$3)$] A closed graph such that $J_G$ is unmixed.
\end{itemize}
Then $\cR^s(J_G)$ is strongly $F$-regular.
\end{theoremx}

As a byproduct of the proof of these results, we prove that simple connected graphs $G$ such that  $|\Ass(J_G)| = 2$ and simple connected bipartite accessible graphs are weakly closed.

\section{Background}\label{sctn background}
\subsection{Symbolic Powers}

We begin recalling the definition of symbolic powers. Moreover, we state some properties that symbolic powers hold.

\begin{definition}\label{dfn symbolicPowers}
Let $I$ be a radical ideal in a Noetherian ring $R$ and let $\q_1, \dots, \q_m$ be its minimal primes. We define the $n$-th symbolic power of $I$, and denote it by $I^{(n)}$, as
$$ I^{(n)} = \bigcap_{i = 1}^m \q_i^n R_{\q_i} \cap R.$$
\end{definition}

Observe that from Definition \ref{dfn symbolicPowers} we have $\q^{(n)} = \q^n R_{\q} \cap R$ for every prime ideal $\q$, and so, $I^{(n)} = \bigcap_{i = 1}^m \q_i^{(n)}$.

Suppose $I$ is a radical ideal, then for every positive integers $n, m$ the following properties hold:
\begin{itemize}
\item[1)] $I^{(1)} = I$.
\item[2)] $I^n \subseteq I^{(n)}$.
\item[3)] $I^{(n)} I^{(m)} \subseteq I^{(n+m)}$.
\item[4)] If $I$ is prime, then $I^{(n)}$ is $I$-primary.
\end{itemize}

The symbolic Rees algebra of $I$ is denoted by $\cR^s(I)$ and is defined as follows:

$$\cR^s(I) = R \oplus I^{(1)}t \oplus I^{(2)}t^2 \oplus I^{(3)}t^3 \dots \subseteq R[t].$$

Now we state a result due to Eisenbud and Hochster that characterizes symbolic powers in the polynomial case.

\begin{theorem}[{\cite{EisenbudHochster}}]
Let $k$ be a field and let $I$ be a radical ideal of $R = k[x_1, \dots, x_n]$. Then
$$ I^{(n)} = \bigcap_{\m \in A} \m^n, $$
where $A = \{ \m \subseteq R \:|\: I \subseteq \m \text{ and } \m \text{ is a maximal ideal }\}.$
\end{theorem}

\subsection{Methods in Prime Characteristic}

Throughout this section $R$ denotes a  commutative Noetherian ring with one of prime characteristic $p$. Let $F:R \to R$ be defined by $r \mapsto r^p$. $F$ is a ring homomorphism and is called the Frobenious morphism. We denote the composition of $F$ with itself $e$ times by $F^e$, for any positive integer $e$. Note that we can view $R$ as an $R$-mod via $F^e$, and in this case, we denote it by $F^e_{*}R$. Observe that $F:R \mapsto F^e_*R$ is a map of $R$-modules. For any $R$-mod map $\psi: M \to N$, we say that $\psi$ is a split monomorphism, or that $\psi$ splits over $M$, if there is an $R$-mod map $\phi: N \to M$ such that $\phi \circ \psi$ is the identity map on $M$. In such case we call $\phi$ a splitting of $\psi$. Notice that if $\psi$ is a split monomorphism, it is indeed a monomorphism. We say that $R$ is $F$-finite if $F_*R$ is a finitely generated $R$-module. If $R$ is $F$-finite, then any finitely generated $R$-algebra is also $F$-finite. When $R$ is reduced, we denote the ring of $p^e$-th roots of $R$ by $R^{1/{p^e}}$, which is a ring extension of $R$. We say that $R$ is $F$-pure if for every $R$-module $M$, the induced map $R \otimes M \to F_* R \otimes M$ is injective. We say that $R$ is $F$-split if $F:R \to F_*R$ is a split monomorphism. If $R$ is reduced we can identify $F^e$ with the inclusion $R \subseteq R^{1/p^e}$ for any $e \geq 1$, and so, $R$ is $F$-split if and only if $R \subseteq R^{1/p^e}$ is a split monomorphism. When $R$ is $F$-finite, $R$ is $F$-pure if and only if $R$ is $F$-split. If $R$ is a domain, we say that $R$ is strongly $F$-regular if for every non zero $r \in R$, there is $e \geq 1$ such that the $R$-linear map $R \to R^{1/{p^e}}$, defined by $1 \mapsto c^{1/{p^e}}$, splits over $R$.

The following result is a criterion that determines if a ring is strongly $F$-regular.

\begin{theorem}[{\cite{HochsterHuneke}}]\label{thm HochsterHunekeCriterion sFr}
Let $R$ be a an $F$-finite Noetherian domain of prime characteristic $p$. Let $c$ be a non zero element of $R$ such that $R_c$ is strongly $F$-regular. Then, $R$ is strongly $F$-regular if and only if the map $R \to R^{1/p^e}$ defined by $1 \mapsto c^{1/p^e}$ splits over $R$ for some $e$. 
\end{theorem}

Now we define what does it mean for an ideal to be symbolic $F$-split.

\begin{definition}[{\cite{destefani2021blowup}}]
Let $\{I_n\}_{n \geq 0}$ be a sequence of ideals in a reduced ring $R$. We say that $\{I_n\}_{n \geq 0}$ is an $F$-split filtration if the following holds:
\begin{itemize}
\item[i)]   $I_0 = R$.
\item[ii)]  $I_{n+1} \subseteq I_{n}$, for every $n \geq 0$.
\item[iii)] $I_n I_m \subseteq I_{n + m}$, for every $n, m \geq 0$.
\item[iv)] There is a splitting $\phi: R^{1/p} \to R$, of $R \subseteq R^{1/p}$, such that $\phi((I_{np+1})^{1/p}) \subseteq I_{n+1}$ for every $n \geq 0$.
\end{itemize}

Let $I$ be an ideal of $R$. We say that $I$ is symbolic $F$-split if $\{I^{(n)}\}_{n \geq 0}$ is an $F$-split filtration.
\end{definition}

It turns out that if $\{I_n\}_{n \geq 0}$  is an $F$-split filtration, then $R/I_1$ is $F$-split and $I_1$ is radical.

\subsection{Binomial Edge Ideals}

\begin{definition}
A simple graph $G$ is an undirected graph which has no loops, that is, there are no edges between a vertex and itself, and it has no more than one edge between each pair of vertices.
\end{definition}

\begin{definition}[{\cite{MR2669070}}]
Let $G$ be a simple graph with vertex set $[d] = \{1, \dots, d\}$ and edges $E$. The binomial edge ideal associated to $G$ is the ideal $J_G$ in the polynomial ring $R = k[x_1, \dots, x_d,$ $y_1, \dots, y_d]$ where $k$ is a field, defined in the following way.
$$ J_G = (f_{i,j} \:|\: \{i, j\} \in E),$$
where $f_{i,j}= x_i y_j - x_j y_i$.
\end{definition}

Let $G$ be a simple graph with vertex set $[d]$. Let $S \subseteq [d]$ and let $T = [d] \setminus  S$. Let $G_T$ be the induced subgraph of $G$ on the vertex set $T$. We also denote $G_T$ by $G \setminus S$. Let $c_{G}(S)$ be the amount of connected components of $G_T$, if there is no risk of confusion, we write $c(S)$ instead. Let $G_{T_1}, \dots, G_{T_{c(S)}}$ be the connected components of $G_T$. Let $\tilde{G}_{T_i}$ be the complete graph on the vertex set of $G_{T_i}$. We define 
$$\p_{S} = J_{\tilde{G}_{T_1}} + \dots + J_{G_{T_{c(S)}}} + (x_i, y_i \:|\: i \in S).$$

The ideal $\p_S$ is a prime ideal of $R$ and it has height $|S| + d - c(S)$ \cite{MR2669070}.

Let $I_2(X) = (f_{i,j} \:|\: 1 \leq i < j \leq d)$. Observe that $\p_\emptyset = I_2(X)$ if $G$ is connected.

The following proposition is a criterion that allows us to determine if a prime ideal $\p_S$ is a minimal prime of $J_G$.

\begin{proposition}[{\cite{MR2669070}}]\label{prop CriteriaForMinPrimesBEI}
Let $G$ be a connected simple graph with vertex set $[d]$ and let $S \subseteq [d]$. Then, $\p_S$ is a minimal prime of $J_G$ if and only if any of the following conditions holds:
\begin{itemize}
\item[$1)$] $S = \emptyset$.
\item[$2)$] $c(S \setminus \{s\}) < c(S)$ for every $s \in S$.
\end{itemize}
\end{proposition}

We say that $S$ is a cut set of the connected graph $G$ with vertex set $[d]$ if $S$ is such that $\p_S$ is a minimal prime of $J_G$.

\section{Criteria for Symbolic F-Splitness}\label{sctn criterionsSymbolicFsplitness}

We begin this section with a remark that is used during the proof of Proposition \ref{prop SymbolicColonInterEquality}. We also give a proof for this remark since we didn't find one in the literature.

\begin{remark}\label{rmrk AssColonContainment}
Let $I, J$ be ideals of a ring $R$ such that $I = (f_1, \dots, f_m)$. Then $$\Ass(J:I) \subseteq \Ass(J).$$
\end{remark}

\begin{proof}
First we show that the proposition is true when $m = 1$. Let $f = f_1$. We want to show that $\Ass(J:f) \subseteq \Ass(J)$. Consider the $R$-mod homomorphism given by 
\begin{align*}
	R/(J:f) & \to R/J\\
	[r] & \mapsto f[r].
\end{align*}
This morphism is injective, and hence, $\Ass(J:f) \subseteq \Ass(J)$. Now, we proceed to the case $m \geq 2$. Since $I = (f_1, \dots, f_m)$, we have that $(J:I) = \bigcap_{i=1}^{m}(J:f_i)$. Consider the $R$-mod homomorphism given by
\begin{align*}
	R/(\bigcap_{i=1}^{m}(J:f_i)) & \to \bigoplus_{i=1}^{m}R/(J:f_i)\\
	[r] & \mapsto ([r],\dots,[r]).
\end{align*}
This map is injective and thus, 
$$\Ass(J:I) = \Ass(\bigcap_{i=1}^{m}(J:f_i)) \subseteq \bigcup_{i=1}^{m}\Ass(J:f_i) \subseteq \Ass(J)$$
\end{proof}

\begin{proposition}\label{prop SymbolicColonInterContainment} 
Let $J$ be a radical ideal in a regular domain of prime characteristic $p$. Let $\p_1, \dots, \p_m$ be the minimal primes of $J$. Then
$$\bigcap_{i = 1}^m ((\p_i^{(a)})^{[p^e]}:\p_i^{(b)}) \subseteq (J^{(a)})^{[p^e]}:J^{(b)} $$
for every $a,b,e \geq 1$.
\end{proposition}

\begin{proof}
Let $x \in  \bigcap_{i = 1}^m ((\p_i^{(a)})^{[p^e]}:\p_i^{(b)})$. We need to show that $xy \in (J^{(a)})^{[p^e]}$ for any $y \in J^{(b)} = \bigcap_{i = 1}^m \p_i ^{(b)}$. For such $y$, we know that $xy \in (\p_i^{(a)})^{[p^e]}$ for every $i$, so it suffices to prove that $ \bigcap_{i = 1}^m (\p_i^{(a)})^{[p^e]} \subseteq (J^{(a)})^{[p^e]}$.
Since Frobenious powers commutes with finite intersections, we have the equality.
\end{proof}

\begin{lemma}\label{lemma PrimarySymbolicColon}
Let $R$ be a regular domain of prime characteristic $p$ and let $\p \neq 0$ be a prime ideal of $R$. Then $(\p^{(a)})^{[p^e]}:\p^{(b)}$ is $\p$-primary, for every $a, b, e \geq 1$ such that $ap^e > b$.
\end{lemma}

\begin{proof}
Since $R$ is Noetherian, we can write $\p^{(b)} = (f_1,\dots,f_r)$ for some $f_i \in R$. 

First we show that $f_i \not\in (\p^{(a)})^{[p^e]}$ for some $i$. We proceed by contradiction, suppose that $f_i  \in (\p^{(a)})^{[p^e]}$ for every $i$. This implies that $(\p^{(a)})^{[p^e]}:\p^{(b)} = R$. By localizing and completing, we get that
$(\m^{a})^{[p^e]}:\m^{b} = \widehat{R}_{\p}$, where $\m$ is the maximal ideal of $\widehat{R}_{\p}$, a power series ring over a field. Hence, $\m^{b} \subseteq (\m^{a})^{[p^e]} \subseteq \m^{ap^e} \subsetneq \m^b$ , which is a contradiction. Therefore there is an $i$ such that $f_i \not\in (\p^{(a)})^{[p^e]}$. 

Since $(\p^{(a)})^{[p^e]}$ is $\p$-primary and $f_i \not\in (\p^{(a)})^{[p^e]}$ for some $i$, we have that $(\p^{(a)})^{[p^e]}:f_i$ is $\p$-primary \cite[Lemma 4.4]{Atiyah}. Since $(\p^{(a)})^{[p^e]}:\p^{(b)} = \bigcap_{j = 1}^r ((\p^{(a)})^{[p^e]}:f_j)$, we conclude that $(\p^{(a)})^{[p^e]}:\p^{(b)}$ is $\p$-primary.
\end{proof}

Before the next result, we state the following remark.

\begin{remark}\label{rmrk IdealContainmentAssociatedPrimes}
Ideal containment can be verified locally in the following sense. If $I, J$ are ideals of $R$ such that $I_{\p} \subseteq J_{\p}$ for every $\p \in \Ass(J)$, then $I \subseteq J$.
\end{remark}

\begin{proposition}\label{prop SymbolicColonInterEquality} 
Let $J$ be a radical ideal in a regular domain of prime characteristic $p$. Let $\p_1, \dots, \p_m$ be the minimal primes of $J$. Then
$$\bigcap_{i = 1}^m ((\p_i^{(a)})^{[p]}:\p_i^{(b)}) = (J^{(a)})^{[p]}:J^{(b)} $$
for every $a,b \geq 1$.
\end{proposition}

\begin{proof}
Let $a, b \geq 1$. We proceed by double containment.

From Proposition \ref{prop SymbolicColonInterContainment}, we already know that
$$\bigcap_{i = 1}^m ((\p_i^{(a)})^{[p]}:\p_i^{(b)}) \subseteq (J^{(a)})^{[p]}:J^{(b)}. $$

Now we proceed to prove that $$ (J^{(a)})^{[p]}:J^{(b)} \subseteq \bigcap_{i = 1}^m ((\p_i^{(a)})^{[p]}:\p_i^{(b)}). $$

In order to prove this, we proceed to verify the containment locally as stated in Remark \ref{rmrk IdealContainmentAssociatedPrimes}.

Observe that 
\begin{eqnarray*}
\Ass(\bigcap_{i = 1}^m ((\p_i^{(a)})^{[p]}:\p_i^{(b)})) 
& \subseteq & \bigcup_{i=1}^{m}\Ass((\p_i^{(a)})^{[p]}:\p_i^{(b)}) \\
& \subseteq & \bigcup_{i=1}^{m}\Ass((\p_i^{(a)})^{[p]}), \text{by Remark \ref{rmrk AssColonContainment}}  \\
& = &  \bigcup_{i=1}^{m}\Ass(\p_i^{(a)}), \text{because $R$ is regular} \\
& = & \{ \p_1, \dots, \p_m \}, \text{because $\p_i^{(a)}$ is $\p_i$-primary}\\
& = & \Ass(J).
\end{eqnarray*} 

Since $((J^{(a)})^{[p]}:J^{(b)})R_{\p_i} = ((\p_i^{(a)})^{[p]}:\p_i^{(b)})R_{\p_i}$ for every $i$, we conclude that $$(J^{(a)})^{[p]}:J^{(b)} \subseteq \bigcap_{i = 1}^m ((\p_i^{(a)})^{[p]}:\p_i^{(b)}).$$

Thus, $$\bigcap_{i = 1}^m ((\p_i^{(a)})^{[p]}:\p_i^{(b)}) = (J^{(a)})^{[p]}:J^{(b)}. $$
\end{proof}

\begin{corollary}\label{cor AssColonEqualsAssIdeal}
Let $J$ be a non zero radical ideal in a regular domain of prime characteristic $p$. Then
$$\Ass((J^{(a)})^{[p]}:J^{(b)}) = \Ass(J),$$
for every $a,b \geq 1$ such that $ap > b$.
\end{corollary}

\begin{proof}
Let $a,b \geq 1$ such that $ap > b$. From Proposition \ref{prop SymbolicColonInterEquality} , we know that $(J^{(a)})^{[p]}:J^{(b)} = \bigcap_{i = 1}^m ((\p_i^{(a)})^{[p]}:\p_i^{(b)})$, where $\p_1, \dots, \p_m$ are the minimal primes of $J$, where each $\p_i \neq 0$ since $J \neq 0$. Lemma \ref{lemma PrimarySymbolicColon} implies that $(\p_i^{(a)})^{[p]}:\p_i^{(b)}$ is $\p_i$-primary for every $i$.

Thus, $\bigcap_{i = 1}^m ((\p_i^{(a)})^{[p]}:\p_i^{(b)})$ is a minimal primary decomposition for $(J^{(a)})^{[p]}:J^{(b)}$. This decomposition has the same associated primes than the associated primes of $J$.
\end{proof}

\begin{proposition}\label{prop ColonIdealContainment}
Let $J$ be a radical ideal in a regular domain of prime characteristic $p$. Then
$$ (J^{(a+1)})^{[p]}:J^{(b+p)} \subseteq (J^{(a)})^{[p]}:J^{(b)},$$
for every $a,b \geq 1$.
\end{proposition}

\begin{proof}
In order to prove the containment we proceed as in Remark \ref{rmrk IdealContainmentAssociatedPrimes}. Thus we show the containment holds when we localize at each of the associated primes of $(J^{(a)})^{[p]}:J^{(b)}$. Since $\Ass((J^{(a)})^{[p]}:J^{(b)}) \subseteq \Ass(J)$, we show that the containment holds at each of the associated primes of $J$. 

Let $\p \in \Ass(J)$. We proceed to prove that 
$$ ((J^{(a+1)})^{[p]}:J^{(b+p)}) R_{\p} \subseteq ((J^{(a)})^{[p]}:J^{(b)}) R_{\p}.$$

Note that the left hand side is equal to $(({\p}R_{\p})^{a+1})^{[p]}:({\p}R_{\p})^{b+p}$, while the right hand side is equal to $(({\p}R_{\p})^{a})^{[p]}:({\p}R_{\p})^{b}$. So we have to prove that 

$$(({\p}R_{\p})^{a+1})^{[p]}:({\p}R_{\p})^{b+p} \subseteq (({\p}R_{\p})^{a})^{[p]}:({\p}R_{\p})^{b}.$$
This containment holds if and only if 
$$(\m^{a+1})^{[p]}:\m^{b+p} \subseteq (\m^{a})^{[p]}:\m^{b},$$
where $\m$ is the maximal ideal of $\widehat{R}_{\p}$, the completion of $R_{\p}$ with respect to ${\p}R_{\p}$.

Let $f \in (\m^{a+1})^{[p]}:\m^{b+p}$. Observe that 
$$(f \m ^{b}) \m^{[p]} \subseteq (f \m^{b}) \m^p = f \m^{b+p} \subseteq (\m^{a+1})^{[p]}.$$ 
Hence, $f \m ^{b} \subseteq (\m^{a+1})^{[p]}:\m^{[p]} = (\m^{a+1}:\m)^{[p]} = (\m^{a})^{[p]}$, because $\widehat{R}_{\p}$ is a power series ring over a field. Thus, $f \in (\m^{a})^{[p]}:\m ^{b}$. We conclude that $(\m^{a+1})^{[p]}:\m^{b+p} \subseteq (\m^{a})^{[p]}:\m^{b}$.
\end{proof}

\begin{corollary}
Let $J$ be a radical ideal in a regular domain of prime characteristic $p$. Then 
$$ (J^{((n+1)+1)})^{[p]}:J^{((n+1)p+1)} \subseteq (J^{(n+1)})^{[p]}:J^{(np+1)}$$
for every $n \geq 0$ and
$$ (J^{(n+1)})^{[p]}:J^{((n+1)p)} \subseteq (J^{(n)})^{[p]}:J^{(np)}$$
for every $n \geq 1$.
\end{corollary}

As a consequence, we improve a criterion for symbolic F-splitness from De Stefani, Monta\~no and N\'u\~nez-Betancourt to only need to check one colon ideal instead of an intersection of finitely many colon ideals \cite[Theorem 5.8, Proposition 5.12]{destefani2021blowup}. Furthermore, the following lemma gives a sufficient condition for symbolic F-splitness.

\begin{lemma}\label{lemma SymbolicFsplitnessCriteriaOnMinimialPrimes}
Let $J$ be a homogeneous radical ideal in a polynomial ring over a field $k$. Let $J = \q_1 \cap \dots \cap \q_l$, with $h_i = \Ht(\q_i)$. Let $\m$ be the irrelevant ideal. If there is an $f$ such that $f \in \bigcap_i \q_i^{h_i}$ and $f^{p-1} \not\in \m^{[p]}$, then $J$ is symbolic $F$-split.
\end{lemma}

\begin{proof}
Suppose there is an $f$ such that $f \in \bigcap_i \q_i^{h_i}$ and $f^{p-1} \not\in \m^{[p]}$.

From previous work \cite[Corollary 5.10]{destefani2021blowup}, we have that $\q_i^{(h_i(p-1))}  \subseteq (\q_i^{(n+1)})^{[p]}:\q_i^{(np+1)}$ for every $n \geq 0$. Since $f \in \q_i^{h_i}$, we have that $f^{p-1} \in \q_i^{h_i(p-1)}  \subseteq \q_i^{(h_i(p-1))} \subseteq (\q_i^{(n+1)})^{[p]}:\q_i^{(np+1)}$, for every minimal prime $\q_i$ and for every $n \geq 0$. Proposition \ref{prop SymbolicColonInterEquality} implies that $f^{p-1} \in (J^{(n+1)})^{[p]}:J^{(np+1)}$ for every $n \geq 0$. Since $f^{p-1} \not \in \m^{[p]}$, where $\m$ is the irrelevant ideal,  we conclude that $J$ is symbolic F-split \cite[Theorem 5.8]{destefani2021blowup}.
\end{proof}

\section{Symbolic F-Splitness of Binomial Edge Ideals}\label{sctn symbolicFsplitnessOfSomeBEI}

Throughout this section, we adopt the following setting. We denote by $k$ an $F$-finite field of prime characteristic $p$. $G$ denotes a simple connected graph on the vertex set $[d]$. We denote by $R$ the polynomial ring $k[x_1, \dots, x_d, y_1, \dots, y_d]$ and by $J_G$ the binomial edge ideal associated to the graph $G$.

Now we show that some families of binomial edge ideals are symbolic F-split. Namely the binomial edge ideals associated to complete multipartite graphs and the binomial edge ideals associated to caterpillar graphs. 

\begin{definition}
A multipartite graph $G$ is a graph such that the vertex set $V$ has a partition of non empty subsets $V_1, \dots, V_l$ such that for every $i$, if $u, v \in V_i$, then $\{u,v\}$ is not an edge of $G$. Furthermore, if for every $i,j$ such that $i \neq j$ we have that $\{u,v\}$ is an edge of $G$ for every $u \in V_i, v \in V_j$, we say that $G$ is multipartite complete. The sets $V_i$ are called the $i$-th part of $G$.
\end{definition}

\begin{figure}[hbt!] \centering \begin{tikzpicture} \newcommand*\pointspxMvw{553.5996924298968/357.702050425149/0/1,838.2817343756791/358.71074124792517/1/4,837.7751100998131/455.9631529693903/2/5,698.8658339028367/614.3752527349557/3/6,554.8930695342559/405.88665532240543/4/2,555.536407222115/453.3146644873013/5/3}
          \newcommand*\edgespxMvw{1/0,2/0,3/0,3/1,3/2,4/1,4/2,4/3,5/1,5/2,5/3}
          \newcommand*\scalepxMvw{0.02}
          \foreach \x/\y/\z/\w in \pointspxMvw {
          \node (\z) at (\scalepxMvw*\x,-\scalepxMvw*\y) [circle,draw,inner sep=3pt] {$\w$};
          }
\foreach \x/\y in \edgespxMvw {
          \draw (\x) -- (\y);
          }
       \end{tikzpicture} \caption[Caption for LOF]{A complete multipartite graph. Note that $V_1 = \{1,2,3\}$, $V_2 = \{4,5\}$ and $V_3 = \{6\}$.\protect\footnotemark} \label{fig:M1} \end{figure}
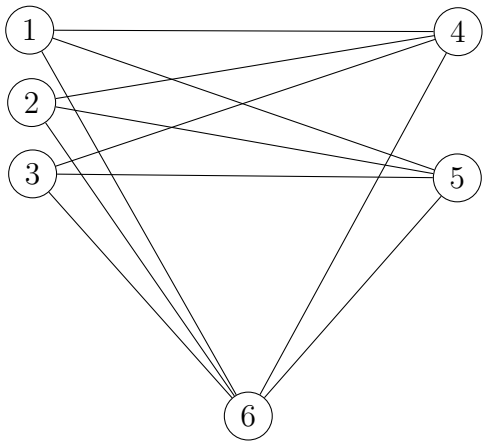

\footnotetext{This image was generated using the package Visualize \cite{VisualizeSource} in Macaulay2 \cite{M2}.}

\begin{theorem}
Let $G$ be a complete multipartite graph. Then $J_G$ is symbolic F-split.
\end{theorem}

\begin{proof}
Let $V_1, \dots, V_l$ be the parts of $G$. From previous work by Ohtani \cite[Lemma 2.2]{Ohtani}, we know that the minimal primes of $J_G$ are the ideals of the form $\p_S$ where $S = \emptyset$ or $S \neq \emptyset$ and $G \setminus S$ is formed by all the vertices of one of the parts of $G$. 

Let $f = y_1 f_{1,2} f_{2,3} \dots f_{d-1,d} x_d$. Let $\p_S$ be a minimal prime of $J_G$ and let $h = \Ht(\p_S)$. We now prove that $f \in \p_S^{h}$ for each minimal prime $\p_S$ of $J_G$. We consider the cases $S = \emptyset$ and $S \neq \emptyset$.

Suppose $S = \emptyset$.  Then $f_{i,i+1} \in \p_S$ for every $i \in \{1, \dots, d - 1\}$. Hence, $f \in \p_S^{d - 1}$. Since $\Ht(\p_S) = |S| + d - c(S) = 0 + d - 1 = d - 1$, we conclude that $f \in \p_S^{h}$. 

Suppose $S \neq \emptyset$. Let $S = \{s_1, s_2, \dots, s_m \}$, where $s_1 < s_2 < \dots < s_m$. For each $i \in \{1, 2, \dots, m-1 \}$, we define the following elements

$$g_0 = \twopartdef { y_1 } {s_1 = 1} {f_{s_1 - 1, s_1}} {s_1 > 1,}$$

$$g_i = \twopartdef { f_{s_i, s_i + 1} } {s_i + 1 = s_{i+1}} {f_{s_i, s_i + 1}f_{s_{i+1} - 1, s_{i+1}}} {s_i + 1 < s_{i+1},}$$
and
$$g_m = \twopartdef { x_d } {s_m = d} {f_{s_m,s_m + 1}} {s_m < d.}$$

Note that if $i \neq j$, then $g_i \neq g_j$. Let $g = g_0 g_1 \dots g_m$. We show that $g \in \p_S^{h}$. We proceed by cases.

First, consider $g_0$. If $s_1 = 1$, then $y_1 \in \p_S$, and if $s_1 > 1$, then $f_{s_1 - 1, s_1} \in \p_S$. Thus, $g_0 \in \p_S$.

Now, consider $g_i$ with $i \in \{1, 2, \dots, m-1 \}$. If $s_i + 1 = s_{i+1}$, then $f_{s_i,s_i + 1} \in  \p_S^2$. If $s_i + 1 < s_{i+1}$, then $f_{s_i, s_i + 1} \neq f_{s_{i+1} - 1, s_{i+1}}$ and  $f_{s_i, s_i + 1}, f_{s_{i+1} - 1, s_{i+1}} \in \p_S$. Thus $f_{s_i, s_i + 1}f_{s_{i+1} - 1, s_{i+1}} \in \p_S^2$, and so $g_i \in \p_S^2$.

Lastly, consider $g_m$. If $s_m = d$, then $x_d \in \p_S$, and if $s_m < d$, then $f_{s_m,s_m + 1} \in \p_S$. Thus, $g_m \in \p_S$. 

This implies that $g = g_0 ( \prod_{i = 1}^{m-1}g_i ) g_m \in \p_S((\p_S^2)^{m-1})\p_S = \p_S^{2m}$. By construction $g$ divides $f$, and hence $f \in \p_S^{2m}$. Recall that $h = \Ht(\p_S) = |S| + d - c(S) = m + d - (d - m)= 2m$. Thus, $f \in \p_S^h$. We conclude that $f \in \p_S^h$ for each minimal prime $\p_S$ of $J_G$.

Since $f^{p-1} \not \in \m^{[p]}$, where $\m$ is the irrelevant ideal, Lemma \ref{lemma SymbolicFsplitnessCriteriaOnMinimialPrimes} implies that $J_G$ is symbolic F-split.
\end{proof}

Othani proved that the symbolic powers and ordinary powers of binomial edge ideals of complete multipartite graphs are the same \cite{Ohtani}, hence the powers of $J_G$ form an F-split filtration. This implies that the Rees algebra $\cR(J_G)$ and the associated graded ring $\gr(J_G)$ are F-split \cite[Theorem 4.7]{destefani2021blowup}.

\begin{corollary}
Let $G$ be a complete multipartite graph. Then $\cR(J_G)$ and $\gr(J_G)$ are $F$-split.
\end{corollary}

\begin{definition}\label{dfn CaterpillarGraph}
A caterpillar graph is a graph $G$ such that it has an induced path graph $P$ on the vertices $v_1, \dots, v_l$ of $G$, where $\{v_i, v_{i+1}\}$ is an edge for $i = 1, 2, \dots, l-1$, and for any other vertex $u$ of $G$ such that $u$ is not a vertex of $P$, we have that there is a unique $v$ with $v \in V(P) \setminus \{v_1, v_l \}$ such that $\{u,v\} \in E(G)$.
\end{definition}

\begin{remark}\label{rmrk CaterpillarGraph}
Observe that for a caterpillar graph $G$ on $[d]$, we can label its vertices in the following way. Let $\{v_1, \dots, v_l\}$ be as in Definition \ref{dfn CaterpillarGraph} and let $X_i$ be the set of vertices of $G \setminus P$ connected uniquely to $v_i$. Let $a_i = |X_i|$. Note that $v_1$ and $v_l$ are uniquely connected to $v_2$ and $v_{l-1}$ respectively in $G$, and thus, $a_1 = a_l = 0$.
First, let $v_1 = 1$.
For $i \geq 1 $, let $v_{i + 1} = v_{i} + a_{i} + 1$. Lastly, label the vertices in $X_i$ by $v_i + 1, v_i + 2, \dots, v_i + a_i$, for $1 < i < l$. 
As a consequence of this labeling, we have that $v_l = d$.
\end{remark}

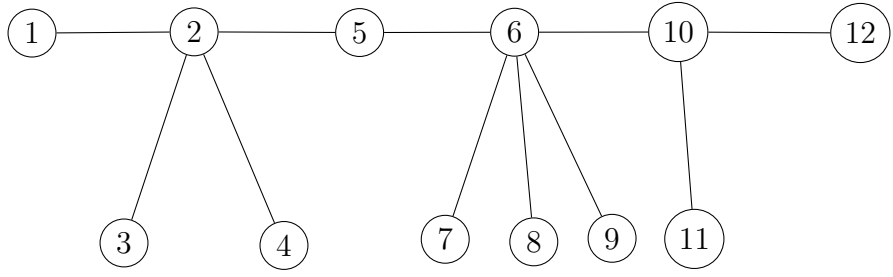
\begin{figure}[hbt!] \centering \begin{tikzpicture}
         \newcommand*\pointsKhSZJ{516.0312250524278/384.9307020943521/0/1,623.7694974631671/384.11792825402614/1/2,577.1709816619907/524.5179866084259/2/3,683.4824779389137/526.1782204630771/3/4,733.7622433615339/384.5747437548659/4/5,836.7616521800276/384.35891420199357/5/6,790.640783666441/522.0649380016152/6/7,849.4943775286632/523.7763824132367/7/8,901.5338515548419/521.4181809026433/8/9,945.5194394783839/384.20934250894067/9/10,956.1948827821664/521.8201938209186/10/11,1066.5098925393693/384.86675056207616/11/12}
          \newcommand*\edgesKhSZJ{1/0,2/1,3/1,4/1,5/4,6/5,7/5,8/5,9/5,10/9,11/9}
          \newcommand*\scaleKhSZJ{0.02}
          \foreach \x/\y/\z/\w in \pointsKhSZJ {
          \node (\z) at (\scaleKhSZJ*\x,-\scaleKhSZJ*\y) [circle,draw,inner sep=3pt] {$\w$};
          }
\foreach \x/\y in \edgesKhSZJ {
          \draw (\x) -- (\y);
          }
      \end{tikzpicture}\caption[Caption for LOF]{A caterpillar graph which is labeled in the way stated in Remark \ref{rmrk CaterpillarGraph}.\protect\footnotemark} \label{fig:M2} \end{figure}

\footnotetext{This image was generated using the package Visualize \cite{VisualizeSource} in Macaulay2 \cite{M2}.}
    
\begin{theorem}
Let $G$ be a caterpillar graph. Then $J_G$ is symbolic F-split.
\end{theorem}

\begin{proof}
Let $G$ be a caterpillar graph on $[d]$. Consider its path graph $P = (V(P),E(P))$ given as in Definition \ref{dfn CaterpillarGraph} and label the vertices of $G$ as stated in Remark \ref{rmrk CaterpillarGraph}. Let $X_i$ and $a_i$ be as in Remark \ref{rmrk CaterpillarGraph}.

If $\p_S$ is a minimal prime of $J_G$, then $S \subseteq V(P)\setminus \{1,d\}$ \cite[Theorem 20]{SharifanJavanbakht}.

Let $f = y_1 f_{1,2} f_{2,3} \dots f_{d-1,d} x_d$. Let $\p_S$ be a minimal prime of $J_G$ and let $h = \Ht(\p_S)$. Now we prove that $f \in \p_S^{h}$ for each $S$ such that $S \subseteq V(P)\setminus \{1,d\}$. We consider the cases $S = \emptyset$ and $S \neq \emptyset$.

Suppose $S = \emptyset$. Then $f_{i,i+1} \in \p_S$ for every $i \in [d-1]$. Hence, $f \in \p_S^{d - 1}$. Since $\Ht(\p_S) = |S| + d - c(S) = 0 + d - 1 = d - 1$, we conclude that $f \in \p_S^{h}$. 

Suppose $S \neq \emptyset$. Let $S = \{s_1, s_2, \dots, s_m\}$ such that $1 < s_1 < s_2 < \dots < s_m < d$.

Observe that 
$$\p_S = (x_s, y_s \:|\: s \in S) + I_2(X_{[1, s_1 - 1]}) + \sum_{i = 1}^{m-1} I_2(X_{[s_i + a_{s_i} + 1, s_{i+1} - 1]}) + I_2(X_{[s_m + a_{s_m} + 1, d]}),$$

where for any integers $i, j$, $I_2(X_{[i,j]})$ is the ideal of $R$ generated by all $2 \times 2$ minors of the matrix $X_{[i,j]} = 
\begin{bmatrix}
x_i & x_{i + 1} &  \dots & x_j\\
y_i & y_{i + 1} &  \dots & y_j
\end{bmatrix}$ 
whenever $i < j$, and is the zero ideal of $R$ otherwise.

Recall that $s_1 \geq 2$. Let 
$$g_0 = \twopartdefo
{\prod_{j = 1}^{s_1 - 2} f_{j, j + 1}} {s_1 \geq 3}
{1} $$

For $i \in [m-1]$, we define $t_i = s_i + a_{s_i} + 1$ and $l_i = s_{i+1} - (s_i + a_{s_i} + 1) = s_{i+1} - t_i$. Note that $l_i \geq 0$ by construction.

Let
$$g_i = \twopartdefo
{\prod_{j = 1}^{l_i - 1} f_{t_i + j - 1, t_i + j}} {l_i \geq 2}
{1} $$

We know that $d -  (s_m + a_{s_m} + 1) \geq 0$.
Let 
$$g_m = \twopartdefo
{\prod_{j = 1}^{d - (s_m + a_{s_m} + 1)} f_{s_m + a_{s_m} + j, s_m + a_{s_m} + j + 1}} {d -  (s_m + a_{s_m} + 1) \geq 1}
{1} $$

Let
$$ g = 
g_0
(\prod_{i = 1}^{m - 1} g_i)
g_m
\prod_{i = 1}^{m } f_{s_i - 1, s_i} f_{s_i, s_i + 1}. 
$$

We now show that $g \in \p_S^h$. First, we compute $h$. Since $h = \Ht(\p_S) = |S| + d - c(S)$, we need to compute $c(S)$ first. In order to do this, let
$$ Y_i = \threepartdef 
{{[d]} \setminus \{s_1, s_1 + 1, \dots, d\}} {i = 0} 
{{[d]} \setminus (\{1,2,\dots,s_i + a_{s_i}\} \cup \{s_{i+1}, s_{i+1} + 1, \dots, d \} )} {i \in [m-1]}
{{[d]} \setminus \{1, 2, \dots, s_m + a_{s_m}\}  } {i = m.} 
$$

Note that the $Y_j's$ are disjoint by construction and that 
$$ |Y_i| = \threepartdef 
{s_1 - 1} {i = 0}
{l_i} {i \in [m-1]}
{d - (s_m + a_{s_m})} {i = m.} 
$$

Observe that $Y_0$ and $Y_m$ are connected components of $G \setminus S$. Similarly, the singletons formed by the members of each $X_{s_i}$ are connected components of $G \setminus S$. Lastly, for $i \in [m-1]$, we have that $Y_i$ is a connected component whenever $Y_i$ it is not the empty set, that is, whenever $|Y_i| = l_i \geq 1$. For $i \in [m-1]$, we define
$$ \delta_i = \twopartdefo
{1} {l_i \geq 1}
{0} 
$$
Thus, $c(S) = 2 + \sum_{i = 1}^m a_{s_i} + \sum_{i = 1}^{m - 1} \delta_i$ and so,
$h = m + d - (2 + \sum_{i = 1}^m a_{s_i} + \sum_{i = 1}^{m - 1} \delta_i)$.

Observe that for $i \in [m - 1]$ we have that

$$g_i \in \threepartdef
{\p_S^{l_i - 1}} {l_i \geq 2}
{\p_S^0 = \p_S^{l_i - 1}} {l_i = 1}
{\p_S^0 = \p_S^{l_i}} {l_i = 0.}
$$

Thus, we have for $i \in [m-1]$ that

$$g_i \in \twopartdefo 
{\p_S^{l_i - 1}} {l_i \geq 1}
{\p_S^{l_i}} 
$$

This implies that $g_i \in \p_S^{l_i - \delta_i}$, for $i \in [m-1]$.

Observe that $g_0 \in \p_S^{s_1 - 2}$ and that $g_m \in \p_S^{d - (s_m + a_{s_m}+1)}$. Thus, $g \in \p_S^b$ where 
\begin{eqnarray*}
b 
& = & (s_1 - 2) + \sum_{i = 1}^{m - 1} (l_i - \delta_i) + (d - (s_m + a_{s_m} + 1)) + 2m\\
& = & (s_1 - 2) + \sum_{i = 1}^{m - 1} l_i + (d - (s_m + a_{s_m} + 1)) + 2m - \sum_{i = 1}^{m - 1} \delta_i \\
& = & (s_1 - 2) + \sum_{i = 1}^{m - 1} (s_{i+1} - (s_i + a_{s_i} + 1)) + (d - (s_m + a_{s_m} + 1)) + 2m - \sum_{i = 1}^{m - 1} \delta_i \\
& = & -2 - \sum_{i = 1}^m (a_{s_i} + 1) + d + 2m - \sum_{i = 1}^{m - 1} \delta_i \\
& = & -2 - \sum_{i = 1}^m a_{s_i} + d + m - \sum_{i = 1}^{m - 1} \delta_i \\
& = & m + d - (2 + \sum_{i = 1}^m a_{s_i} + \sum_{i = 1}^{m - 1} \delta_i) \\
& = & h.
\end{eqnarray*}

Therefore $b = h$ and $g \in \p_S^h$. This implies that $f \in \p_S^{h}$ since $g$ divides $f$. 

Since $f^{p-1} \not \in \m^{[p]}$, where $\m$ is the irrelevant ideal, Lemma \ref{lemma SymbolicFsplitnessCriteriaOnMinimialPrimes} implies that $J_G$ is symbolic F-split.
\end{proof}

Jahani, Bayati and Rahmati proved that the symbolic powers and ordinary powers of binomial edge ideals of complete caterpillar graphs are the same \cite{MR4520291}, hence the powers of $J_G$ form an F-split filtration. This implies that the Rees algebra $\cR(J_G)$ and the associated graded ring $\gr(J_G)$ are F-split \cite[Theorem 4.7]{destefani2021blowup}.

\begin{corollary}
Let $G$ be a caterpillar tree. Then $\cR(J_G)$  and $\gr(J_G)$ are $F$-split.
\end{corollary}

\begin{definition}
Let $G$ be a simple graph with $d$ vertices and edge set $E$. We say that $G$ is closed if there is a labeling of the vertices of $G$ by $[d]$ such that the following condition holds: if $\{ i, k \} \in E$, then  $\{ i, j \} \in E$ and $\{ j, k \} \in E$ for every $j$ such that $i < j < k$.
\end{definition}

\begin{definition}
Let $G$ be a simple graph with $d$ vertices and edge set $E$. We say that $G$ is weakly closed if there is a labeling of the vertices of $G$ by $[d]$ such that the following condition holds: if $\{ i, k \} \in E$, then  $\{ i, j \} \in E$ or $\{ j, k \} \in E$ for every $j$ such that $i < j < k$.
\end{definition}

It is known that the binomial edge ideals associated to some closed graphs are symbolic F-split \cite[Proposition 6.39]{destefani2021blowup}. The closed graphs are part of a bigger family of graphs, the weakly closed graphs.
Complete multipartite graphs and caterpillar graphs are weakly closed \cite{MR3882405}. 
It is still not known if the binomial edge ideals of weakly closed graphs are symbolic F-split. In the remaining of this section we study the symbolic $F$-splitness of two more families of graphs. We also explore the weakly closed property for these graphs.

\begin{definition}
For $i \in \{1,2\}$, let $G_i$ be a simple graph with vertex set $V_i$  and edge set $E_i$. Furthermore, suppose $V_1 \cap V_2 = \emptyset$. The join of $V_1$ and $V_2$ is the graph $G$ with vertex set $V = V_1 \cup V_2$ and edge set $E = E_1 \cup E_2 \cup \{ (v_1, v_2) \:|\: v_1 \in V_1, v_2 \in V_2 \}$.
\end{definition}

Sharifan proved that for a connected simple graph $G$, $|\Ass(J_G)| = 2$ if and only if $G$ is the join of a complete graph $G_0$ and a disjoint union of complete graphs $G_1, \dots, G_m$ \cite[Corollary 4.2]{Sharifan2AssociatedPrimes}. Furthermore, the minimal primes of $J_G$ are given by $\p_{\emptyset} = I_2(X)$, and $\p_{V_0}$, where $V_0$ denotes the vertex set of $G_0$ \cite[Lemma 4.1]{Sharifan2AssociatedPrimes}. 

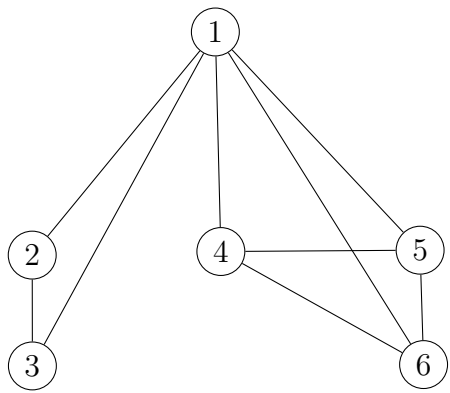
\begin{figure}[hbt!] \centering \begin{tikzpicture}
         \newcommand*\pointsqxcFh{903.8910660689519/344.741876289422/0/1,782.1785150789514/493.14283590268496/1/2,782.2968294394094/566.8585367065896/2/3,907.5494620686293/490.7647664703755/3/4,1039.9655706205815/489.8691036748871/4/5,1042.273364879736/565.9551239213909/5/6}
          \newcommand*\edgesqxcFh{1/0,2/0,2/1,3/0,4/0,4/3,5/0,5/3,5/4}
          \newcommand*\scaleqxcFh{0.02}
          \foreach \x/\y/\z/\w in \pointsqxcFh {
          \node (\z) at (\scaleqxcFh*\x,-\scaleqxcFh*\y) [circle,draw,inner sep=3pt] {$\w$};
          }
\foreach \x/\y in \edgesqxcFh {
          \draw (\x) -- (\y);
          }
      \end{tikzpicture}\caption[Caption for LOF]{A graph $G$ such that $|\Ass(J_G)| = 2$. Note that in this case $G$ is the join of $G_0 = (\{1\}, \emptyset
      )$ and the disjoint union of $G_1$ and $G_2$ where $G_1$ is the complete graph on the vertex set $\{2,3\}$ and $G_2$ is the complete graph on the vertex set $\{4,5,6\}$.\protect\footnotemark} \label{fig:M3} \end{figure}

\footnotetext{This image was generated using the package Visualize \cite{VisualizeSource} in Macaulay2 \cite{M2}.}

\begin{proposition}\label{prop AssG2AreWC}
Let $G$ be a graph such that $|\Ass(J_G)| = 2$. Then $G$ is weakly closed. 
\end{proposition}

\begin{proof}
Since $|\Ass(J_G)| = 2$, we have that $G$ is the join of a complete graph $G_0$ and a disjoint union of complete graphs $G_1, \dots, G_m$. Let $n_i$ be the amount of vertices of $G_i$. We relabel the vertices of $G$ by relabeling the vertices of each $G_i$ as follows: relabel the vertices of $G_0$ by $1, 2, \dots, n_0$ and for $i > 0$, relabel the vertices of $G_i$ by $u_i + 1, u_i + 2, \dots, u_i + n_i$, where $u_i = n_0 + \dots + n_{i - 1}$. 

Now we show that this labeling satisfies the weakly closed property. Let $\{a,c\} \in E(G)$ with $a < c$. Suppose there is a $b$ such that $a < b < c$. 

If $\{a,c\} \in E(G_i)$ for some $i$, then $\{a,b\} \in E(G_i) \subseteq E(G)$, since $G_i$ is complete and $b \in V(G_i)$ by our given labeling. 

Otherwise, $a$ and $c$ belong to $V(G_i)$ and $V(G_j)$ respectively, for some $i,j$ such that $i < j$ because of the given labeling. Since the $G_i's$ are pairwise disjoint for every $i > 0$, we have that $a \in G_0$. If $b \in G_0$, then $\{ a, b\} \in E(G_0) \subseteq E(G)$ since $G_0$ is complete. Otherwise, $b \in E(G_l)$ for some $l > 0$. Since $G$ is the join of $G_0$ and the disjoint union of the graphs $G_1, \dots, G_m$, we have that $\{a, b\} \in E(G)$.

We conclude that $G$ is weakly closed.
\end{proof}

\begin{theorem}\label{thm SymbFsplitAss2}
Let $G$ be such that $|\Ass(J_G)| = 2$. Then $J_G$ is symbolic F-split.
\end{theorem}

\begin{proof}
Let $G$ be a connected simple graph such that $|\Ass(J_G)| = 2$. We know that $G$ is the join of a complete graph $G_0$ and the disjoint union of complete graphs $G_1, \dots, G_m$. Label the vertices of $G_0$, \dots, $G_m$ as stated in the proof of Proposition \ref{prop AssG2AreWC}

Let $f = y_1 f_{1,2} f_{2,3} \dots f_{d-1,d} x_d$. Let $\p_S$ be a minimal prime of $J_G$ and let $h = \Ht(\p_S)$. We proceed to prove that $f \in \p_S^{h}$ for each minimal prime $\p_S$ of $J_G$. We know that the minimal primes of $J_G$ are given by $\p_S$ where $S \in \{\emptyset, V_0\}$, where $V_0$ is the vertex set of $G_0$. Thus, we consider the cases $S = \emptyset$ and $S \neq \emptyset$.

Suppose $S = \emptyset$. Then $f_{i,i+1} \in \p_S$ for every $i \in [d-1]$. Hence, $f \in \p_S^{d - 1}$. Since $\Ht(\p_S) = |S| + d - c(S) = 0 + d - 1 = d - 1$, we conclude that $f \in \p_S^{h}$. 

Suppose $S \neq \emptyset$. Then $S = V_0 = [n_0]$. We have that $h = \Ht(\p_S) = |S| + d - c(S)= n_0 + d - m$.

Let
$$ g = 
x_1
(\prod_{i = 0}^{m} g_i)
f_{n_0, n_0 + 1}
,
$$

where for each $i$,
$$g_i = \twopartdefo
{\prod_{j = 1}^{n_i - 1} f_{j, j + 1}} {n_i \geq 2}
{1} $$

Observe that $x_1, f_{n_0, n_0 + 1} \in \p_S$. Furthermore, notice that $g_0 \in \p_S ^{2(n_0 - 1)}$ and $g_i \in \p_S ^{n_i - 1}$ for $i \in [m]$. Hence, $g \in \p_S ^{b}$, where 
$$b = 2 + 2(n_0 - 1) + \sum_{i = 1}^m (n_i - 1) = n_0 + \sum_{i = 1}^m n_i - m = n_0 + d - m = h.$$

Thus $b = h$ and so, $g \in \p_S^{h}$.  This implies that $f \in \p_S^{h}$ since $g$ divides $f$. 

Since $f^{p-1} \not \in \m^{[p]}$, where $\m$ is the irrelevant ideal, Lemma \ref{lemma SymbolicFsplitnessCriteriaOnMinimialPrimes} implies that $J_G$ is symbolic F-split.
\end{proof}

Jahani, Bayati and Rahmati proved that the symbolic powers and ordinary powers of binomial edge ideals of graphs such that $|\Ass(J_G) = 2|$ are the same \cite{MR4520291}, hence the powers of $J_G$ form an F-split filtration. This implies that the Rees algebra $\cR(J_G)$ and the associated graded ring $\gr(J_G)$ are F-split \cite[Theorem 4.7]{destefani2021blowup}.

\begin{corollary}
Let $G$ be such that $|\Ass(J_G)| = 2$. Then $\cR(J_G)$  and $\gr(J_G)$ are $F$-split.
\end{corollary}

Now we study another family of graphs. We begin with a definition due to Bolognini, Macchia and Strazzanti.

\begin{definition}[{\cite{CMBEIandAccessibleGraphs}}]
We say that a graph $G$ is accessible if $J_G$ is unmixed and the following holds: for every non empty cut set $S$ of $G$, there is an $s \in S$ such that $S \setminus \{s\}$ is a cut set of $G$.
\end{definition}

\begin{definition}\label{def GmGraphs}
Let $m \geq 1$. The graph $G_m$ has vertex set $[2m]$. The edges $E$ of $G_m$ are given by the following rule: $\{a,b\} \in E$ if and only if $a$ is odd, $b$ is even and $a < b$.
\end{definition}

Note that $G_m$ is a connected bipartite graph on the vertex sets $A_m = \{1, 3, \dots, 2m - 1\}$ and $B_m = \{2, 4, \dots, 2m\}$.

In some situations the next labeling of $G_m$ is useful. We call $F_m$ to the graph $G_m$ with the following relabeling of its vertices: if $g$ is a vertex of $G_m$ such that $g$ is even, relabel it as $g - 1$. Otherwise, relabel $g$ as $g + 1$.

We define two operations that arise in a characterization of accessible bipartite graphs which we use afterwards.

\begin{definition}[{\cite{CMBEIandAccessibleGraphs}}]
Let $G, H$ be graphs such that $G$ has a leaf $g$, and $H$ has a leaf $h$. We denote by $(G,g) * (H,h)$ the graph obtained by identifying $g$ with $h$. 
\end{definition}

\begin{definition}[{\cite{CMBEIandAccessibleGraphs}}]
Let $G, H$ be graphs such that $G$ has a leaf $g$ with neighbour $g'$, and $H$ has a leaf $h$ with neighbour $h'$. We denote by $(G,g) \circ (H,h)$ the graph obtained by identifying $g'$ with $h'$ and removing $g$ and $h$. 
\end{definition}

It was proved by Bolognini, Macchia and Strazzanti that if $G$ is a bipartite graph, then the following are equivalent \cite[Theorem 6.1]{BEIofBipartiteGraphs}:

\begin{itemize}
\item[1)] $G$ is accessible.  
\item[2)] 
$G = A_1 * A_2 * \dots * A_m$, where $A_i = F_n$ or $A_i = F_{n_1} \circ F_{n_2} \circ \dots \circ F_{n_r}$ with $n \geq 1$ and $n_j \geq 3$ (this implies that $G$ is traceable \cite[Corollary 6.9]{CMBEIandAccessibleGraphs}).
\item[3)] $J_G$ is Cohen-Macaulay.
\end{itemize}

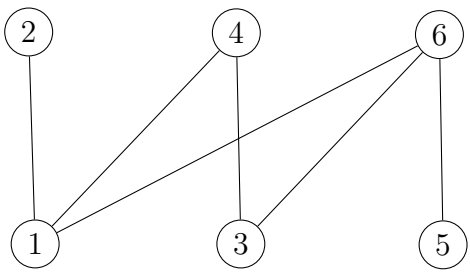
\begin{figure}[hbt!] \centering \begin{tikzpicture}
         \newcommand*\pointscihSE{679.5671503661792/483.9060444380557/0/1,675.2754275593505/342.93019298759145/1/2,813.268901704861/343.4853150870035/2/4,948.2393939763299/344.6754700629197/3/6,816.263878887222/483.94444080064534/4/3,950.6668294026965/484.498076231039/5/5}
          \newcommand*\edgescihSE{1/0,2/0,3/0,4/2,4/3,5/3}
          \newcommand*\scalecihSE{0.02}
          \foreach \x/\y/\z/\w in \pointscihSE {
          \node (\z) at (\scalecihSE*\x,-\scalecihSE*\y) [circle,draw,inner sep=3pt] {$\w$};
          }
\foreach \x/\y in \edgescihSE {
          \draw (\x) -- (\y);
          }
      \end{tikzpicture}\caption[Caption for LOF]{$G_m$ with $m = 3$.\protect\footnotemark} \label{fig:M4} \end{figure}
\footnotetext{This image was generated using the package Visualize \cite{VisualizeSource} in Macaulay2 \cite{M2}.}

Now we show that all accessible bipartite graphs are weakly closed. We start with the following proposition.

\begin{proposition}\label{prop FmGraphsAreWC}
For every $m \geq 1$, the graph $G_m$ is weakly closed.
\end{proposition}

\begin{proof}
Let $m \geq 1$ and consider the graph $G_m$. Suppose that $\{a,b\}$ is an edge of $G_m$ with $a$ odd, $b$ even and $a < b$. Suppose there is a $v$ such that $a < v < b$. If $v$ is odd, then $\{v,b\}$ is an edge of $G_m$ since $b$ is even and $v < b$. If $v$ is even, then $\{a,v\}$ is an edge of $G_m$ since $a$ is odd and $a < v$. This implies that $G_m$ is weakly closed.
\end{proof}

\begin{lemma}\label{lemma CircIsWC}
For $i = 1, 2$, let $H_i$ be a weakly closed graph on the vertex set $[m_i]$. Suppose $H_i$ has a labeling on $[m_i]$ such that:
\begin{itemize}
\item[1)] The labeling fulfils the weakly closed condition.
\item[2)] $\{1,2\}, \{m_i - 1, m_i\} \in E(H_i)$. %with $2 < m_i - 1$.
\item[3)] The vertices $1$ and $m_i$ are leaves.
\end{itemize}
Then $(H_1, m_1) \circ (H_2, 1)$ is weakly closed.
\end{lemma}

\begin{proof}
Let $H = (H_1, m_1) \circ (H_2, 1)$. Observe that $H$ has $m_1 + m_2 - 3$ vertices. We relabel the vertices of $H$ as follows. Let $h \in V(H)$. If $h$ is a vertex of $H_1$, we keep its original label as in $H_1$, otherwise, we relabel $h$ as $h + m_1 - 3$. We claim that this labeling is such that it fulfils the weakly closed condition. We proceed to prove this.

Let $u, v, w \in V(H)$ such that $u < v < w$ and $\{u,w\} \in E(H)$. We consider the following cases:
\begin{itemize}
\item[Case 1:] Suppose $u, w \in V(H_1)$. Hence, $v \in V(H_1)$. Since $H_1$ is weakly closed, then we have that $\{u,v\} \in E(H_1)$ or $\{v,w\} \in E(H_2)$. Thus, $\{u,v\} \in E(H)$ or $\{v,w\} \in E(H)$.
\item[Case 2:] Suppose $u, w \in V(H_2)$. Hence, $v \in V(H_2)$. Since the relabeling of the elements of $H_2$ in $H$ preserves order  and $H_2$ is weakly closed, we have that $\{u,v\} \in V(H_2)$ or $\{v,w\} \in V(H_2)$. Thus, $\{u,v\} \in E(H)$ or $\{v,w\} \in E(H)$.
\item[Case 3:] Suppose $u \in V(H_1)$ and $w \in V(H_2)$. We have that $m_1 - 1$ is the only vertex of $H$ which is also a vertex of $H_1$ and $H_2$ by construction. Hence, this case does not occur. 
\end{itemize}
We conclude that $H$ is weakly closed.
\end{proof}

Now we state the analogous proposition for the $*$ operation. The proof is very similar to that of the $\circ$ operation.

\begin{lemma}\label{lemma StarIsWC}
For $i = 1, 2$, let $H_i$ be a weakly closed graph on the vertex set $[m_i]$. Suppose $H_i$ has a labeling on $[m_i]$ such that:
\begin{itemize}
\item[1)] The labeling fulfills the weakly closed condition.
\item[2)] $\{1,2\}, \{m_i - 1, m_i\} \in E(H_i)$.
\item[3)] The vertices $1$ and $m_i$ are leaves.
\end{itemize}
Then $(H_1, m_1) * (H_2, 1)$ is weakly closed.
\end{lemma}

\begin{proof}
Let $H = (H_1, m_1) * (H_2, 1)$. Observe that $H$ has $m_1 + m_2 - 1$ vertices. We relabel the vertices of $H$ as follows. Let $h \in V(H)$. If $h$ is a vertex of $H_1$, we keep its original label as in $H_1$, otherwise, we relabel $h$ as $h + m_1 - 1$. We claim that this labeling is such that it fulfils the weakly closed condition. We proceed to prove this.

Let $u, v, w \in V(H)$ such that $u < v < w$ and $\{u,w\} \in E(H)$. We consider the following cases:
\begin{itemize}
\item[Case 1:] Suppose $u, w \in V(H_1)$. Hence, $v \in V(H_1)$. Since $H_1$ is weakly closed, then we have that $\{u,v\} \in E(H_1)$ or $\{v,w\} \in E(H_2)$. Thus, $\{u,v\} \in E(H)$ or $\{v,w\} \in E(H)$.
\item[Case 2:] Suppose $u, w \in V(H_2)$. Hence, $v \in V(H_2)$. Since the relabeling of the elements of $H_2$ in $H$ preserves order  and $H_2$ is weakly closed, we have that $\{u,v\} \in V(H_2)$ or $\{v,w\} \in V(H_2)$. Thus, $\{u,v\} \in E(H)$ or $\{v,w\} \in E(H)$.
\item[Case 3:] Suppose $u \in V(H_1)$ and $w \in V(H_2)$. We have that $m_1 $ is the only vertex of $H$ which is also a vertex of $H_1$ and $H_2$ by construction. Hence, this case does not occur. 
\end{itemize}
We conclude that $H$ is weakly closed.
\end{proof}

\begin{remark}
Notice that the proofs of the previous lemmas can be adapted in order to show that $(H_1, m_1) * (H_2, m_2), (H_1, 1) * (H_2, m_2), (H_1, 1) * (H_2, 1), (H_1, m_1) \circ (H_2, m_2), (H_1, 1) \circ (H_2, m_2)$, and $(H_1, 1) \circ (H_2, 1)$ are weakly closed as well. 
\end{remark}

We are ready to prove that bipartite accessible graphs are weakly closed.

\begin{proposition}
Let $G$ be a bipartite graph. If $G$ is accessible, then $G$ is weakly closed. 
\end{proposition}

\begin{proof}
Recall that if $G$ is bipartite and accessible, then $G = A_1 * A_2 * \dots * A_m$, where $A_i = F_n$ or $A_i = F_{n_1} \circ F_{n_2} \circ \dots \circ F_{n_r}$ with $n \geq 1$ and $n_j \geq 3$. Since $F_l$ is just a relabeling of $G_l$, and $G_l$ is weakly closed from Proposition \ref{prop AssG2AreWC}, we have that each $F_l$ is weakly closed. Lemma \ref{lemma CircIsWC} implies that every $A_i$ is weakly closed. Finally, \ref{lemma StarIsWC} implies that $G$ is weakly closed.
\end{proof}

\begin{theorem}\label{thm UnmixedAndTraceableAreSymbolicFsplit}
Let $G$ be a graph such that $J_G$ is unmixed and $G$ is traceable. Then $J_G$ is symbolic $F$-split.
\end{theorem}

\begin{proof}
Let $f = y_1 f_{1,2} f_{2,3} \dots f_{d-1,d} x_d$. Since $G$ is traceable, we relabel the vertices of $G$ in such a way that the path $1, 2, \dots, d$ is part of $G$. Hence, $f \in J_{G}^{d-1}$, and so, $f^{p-1} \in J_{G}^{(p-1)(d-1)} \subseteq J_{G}^{((p-1)(d-1))}$. Since $J_{G}$ is unmixed, $\Ht(J_G) = d - 1$, and $f ^{p-1}\not \in \m^{[p]}$, we conclude that $J_{G}$ is symbolic $F$-split \cite[Corollary 5.10]{destefani2021blowup}.
\end{proof}

\begin{corollary}\label{cor BipartiteAccessibleSymbolicFsplit}
Let $G$ be a bipartite graph. If $G$ is accessible, then $J_{G}$ is symbolic F-split.
\end{corollary}

\begin{proof}
Since $G$ is a connected bipartite accessible graph, it is traceable and $J_G$ is unmixed. It follows from Theorem \ref{thm UnmixedAndTraceableAreSymbolicFsplit} that $J_G$ is symbolic $F$-split.
\end{proof}

\begin{corollary}\label{cor GmSymbolicFsplit}
Let $m \geq 1$. Let $J_{G_m}$ be the binomial edge ideal associated to $G_m$ in the polynomial ring $k[x_1, \dots, x_{2m}, y_1, \dots, y_{2m}]$. Then $J_{G_m}$ is symbolic F-split.
\end{corollary}

\begin{proof}
Since $G_m$ is bipartite and accessible, it follows from Corollary \ref{cor BipartiteAccessibleSymbolicFsplit} that $J_{G_m}$ is symbolic F-split.
\end{proof}

\color{black}

\section{Strong F-Regularity of Blowup Algebras Associated to Binomial Edge Ideals}\label{sctn strongFregularityOfSomeBEI}

In this section we provide examples of symbolic Rees algebras which are strongly $F$-regular. Throughout this section, we adopt the following setting. We denote by $k$ an $F$-finite field of prime characteristic $p$. $G$ denotes a simple connected graph on the vertex set $[d]$. We denote by $R$ the polynomial ring $k[x_1, \dots, x_d, y_1, \dots, y_d]$ and by $J_G$ the binomial edge ideal associated to the graph $G$. 

We begin with a lemma.

\begin{lemma}\label{lemma CriteriaForStrongFregularity}
Let $J$ be a homogeneous radical ideal of $S$, a polynomial ring over the field $k$. Let $\q_1, \q_2, \dots, \q_l$ be the minimal primes of $J$. For $i \in [l]$, let $h_i = \Ht(\q_i)$. Let $\m$ be the irrelevant ideal of $S$. Let $f \in S$ such that $f = cg$, where $c,f \in S$ and $c \neq 0$. Suppose the following holds:
\begin{itemize}
\item[$1)$] $g \in \bigcap_i \q_i^{h_i - 1}$.
\item[$2)$] $f^{p-1} \not\in \m^{[p]}$.
\item[$3)$] $(\cR^s(J))_{c}$ is strongly $F$-regular.
\item[$3)$] $\cR^s(J)$ is Noetherian.
\end{itemize}
Then $\cR^s(J)$ is strongly $F$-regular.
\end{lemma}

\begin{proof}
Suppose 1), 2), 3) and 4) hold. From previous work  \cite[Proposition 5.12]{destefani2021blowup}, we have that $\q_i^{((h_i-1)(p-1))}  \subseteq (\q_i^{(n)})^{[p]}:\q_i^{(np)}$ for every $n \geq 1$. Since $g \in \q_i^{h_i - 1}$, we have that $g^{p-1} \in \q_i^{(h_i-1)(p-1)}  \subseteq \q_i^{((h_i-1)(p-1))} \subseteq (\q_i^{(n)})^{[p]}:\q_i^{(np)}$, for every minimal prime $\q_i$ and for every $n \geq 1$. Proposition \ref{prop SymbolicColonInterEquality} implies that $g^{p-1} \in (J^{(n)})^{[p]}:J^{(np)}$ for every $n \geq 1$. 

Since $f^{p-1} \not \in \m^{[p]}$, we conclude that $\cR^s(J)$ is  F-split \cite[Proposition 5.12]{destefani2021blowup}. This implies that the map $\phi: \cR^s(J) \to \cR^s(J)^{1/p}$, where $\phi(1) = c^{1/p}$, splits. Moreover, since $(\cR^s(J))_{c}$ is strongly $F$-regular, we conclude from Theorem \ref{thm HochsterHunekeCriterion sFr} that $\cR^s(J)$ is strongly $F$-regular.
\end{proof}

For a square free monomial ideal $I$, we know that $\cR^s(I)$ is Noetherian \cite{Lyubeznik} and normal \cite{HerzogHibiTrung}, which implies that $\cR^s(I)$ is strongly $F$-regular. Using Lemma \ref{lemma CriteriaForStrongFregularity}, we give a different proof of this already known fact.

\begin{theorem}[{\cite{HerzogHibiTrung}}]\label{thmExampleMonomialSqFree sFr}
Let $I$ be a square free monomial ideal in $S$, a polynomial ring over a field $k$. Then $\cR^s(I)$ is strongly $F$-regular.
\end{theorem}

\begin{proof}
We proceed by induction on the number of variables of $S$. First we study the base case. Suppose $S = k[x]$. Then $I = (x)$ and $I^{(m)} = I^m$ for every $m$. We have that
$\cR^{s}(I) = \bigoplus_{m \geq 0} I^{(m)}t^m = \bigoplus_{m \geq 0} I^{m}t^m = \bigoplus_{m \geq 0} k(xt)^{m} = k[xt].$

Since $k[xt]$ is strongly $F$-regular, then $\cR^{s}(I)$ is strongly $F$-regular.

Now we proceed with the inductive step. Suppose the statement is true for polynomial rings of $d$ variables. We show that the statement is also true for polynomial rings of  $d+1$ variables. Since $I$ is a square free monomial ideal, we have that $I = \q_1 \cap \dots \cap \q_l$ where every $\q_i$ is a minimal prime of $I$ that can be generated by only variables.

Let $\q \in \{\q_1, \dots, \q_l\}$ and let $h = \Ht(\q)$. Since $h$ is the number of variables that generate $\q$, we can write $\q = (x_{a_1}, \dots, x_{a_h})$ where $\{x_{a_1}, \dots, x_{a_h}\} \subseteq \{x_1, \dots, x_{d+1}\}$ and $a_1 < \dots < a_h$. This implies that $x_{a_2} \dots x_{a_h} \in \q^{h-1}$. Hence, $g = x_2 \dots x_{d+1} \in \bigcap_i \q_i^{h_i - 1}$ where $h_i = \Ht(\q_i)$. Thus, letting $f = x_1 g$ we have that $f^{p-1} \not \in \m^{[p]}$, where $\m$ is the irrelevant ideal.

Now, we show that $(\cR^s(I))_{x_1}$ is strongly $F$-regular. If $x_1 \in I$, then $I_{x_1} = R_{x_1}$ and thus $(\cR^s(I))_{x_1} = R_{x_1}[t]$. Hence, $(\cR^s(I))_{x_1}$ is strongly $F$-regular. Now, suppose $x_1 \not\in I$. Thus, there is at least a $\q_i$ such that $x_1 \not\in \q_i$. Let $J$ be the intersection of the $\q_i's$ such that $x_1 \not\in \q_i$. Observe that $I_{x_1} = J_{x_1}$. Let $\beta = J \cap k[x_2, \dots, x_{d+1}]$. Note that the minimal primary decomposition of $\beta$ is given by $\cap_i \p_i$, where each $\p_i = \q_i \cap k[x_2, \dots, x_{d+1}]$ and $x_1 \not\in \q_i$. This implies that $\p_i k[x_1, \dots, x_{d+1}, x_1^{-1}] = {\q_i}_{x_1}$. As a consequence, we have that $\beta k[x_1, \dots, x_{d+1}, x_1^{-1}] = J_{x_1}$. 

Since each $\p_i$ is an ideal generated by only variables by construction, then $\beta$ is a square free monomial ideal. Thus $\cR^s(\beta)$ is strongly $F$-regular by the induction hypothesis. 

Observe that
$$ (\cR^s(I))_{x_1} \cong \cR^s(\beta)[x_1, x_1^{-1}]. $$

Hence, $(\cR^s(I))_{x_1} $ is strongly $F$-regular.

Finally, we have that $\cR^s(I)$ is Noetherian \cite[Theorem 1]{Lyubeznik}, and so, 
Lemma \ref{lemma CriteriaForStrongFregularity} implies that $\cR^s(I)$ is strongly $F$-regular. Thus, the statement is true by induction.
\end{proof}

We now proceed to prove that binomial edge ideals of complete multipartite graphs have strongly $F$-regular symbolic Rees algebras. We use Lemma \ref{lemma CriteriaForStrongFregularity} to show this. We first prove that graphs $G$ such that $|\Ass(J_G)| = 2$ have strongly $F$-regular symbolic Rees algebras, and then reduce the complete multipartite case to the previous one.

\begin{theorem}\label{thm ass leq 2 is sFr}
Let $G$ be a graph such that $|\Ass(J_G)| = 2$. Then $\cR^s(J_G)$ is strongly $F$-regular.
\end{theorem}

\begin{proof}
Let $f = y_1 f_{1,2} \dots f_{d-1,d} x_d$, $c = y_1$, $g = f/c$. We know that $J_G = I_2(X) \cap \p_{V_0}$, where $V_0 = [n_0]$. We know that $g \in I_2(X)^{d-1} \subseteq I_2(X)^{d-2}$. From the proof of Theorem \ref{thm SymbFsplitAss2}, we know that $g \in \p_{V_0}^ {\Ht(\p_{V_0})- 1}$. Note that $f^{p-1} \not\in \m^{[p]}$, where $\m$ is the irrelevant ideal of $R$. Since $(\cR^s(J_G))_{y_1} = (\cR^s(I_2(X)))_{y_1}$, we have that $(\cR^s(J_G))_{y_1}$ is strongly $F$-regular. Observe that $\cR^s(J)$ is Noetherian since the symbolic powers and the ordinary powers of $J_G$ coincide \cite{MR4520291}. It follows from Lemma \ref{lemma CriteriaForStrongFregularity} that $\cR^s(J_G)$ is strongly $F$-regular.
\end{proof}

\begin{theorem}
Let $G$ be a complete multipartite graph. Then $\cR^s(J_G)$ is strongly $F$-regular.
\end{theorem}

\begin{proof}

Let $f = y_1 f_{1,2} \dots f_{d-1,d} x_d$ and let $g = f/y_1$.
We know that the minimal primes of $J_G$ are $I_2(X), \p_{S_1}, \p_{S_2}, \dots, \p_{S_m}$, where $S_1 = [a]$ for some $a > 1$. Thus $y_1 \in P_{S_i}$ if and only if $i  \neq 1$. This implies that $(\cR^s(J_G))_{y_1} = (\cR^s(I_2(X) \cap \p_{S_1}))_{y_1}$. Observe that $I_2(X) \cap \p_{S_1}$ is a binomial edge ideal. That is, $I_2(X) \cap \p_{S_1} = J_H$, where $H$ is the complete multipartite graph on $[d]$ whose parts are as follows:  $V_1 = [a]$ and $V_i = \{a + i - 1\}$, for $i \in [d + 1 - a] \setminus \{1\}$. By Theorem \ref{thm ass leq 2 is sFr}, we know that $\cR^s(J_H) = \cR^s(I_2(X) \cap \p_{S_1}))$ is strongly $F$-regular. Thus, $(\cR^s(J_G))_{y_1}$ is strongly $F$-regular. As in the proof of Theorem \ref{thm ass leq 2 is sFr}, we have that for any minimal prime $\p$ of $J_G$, $g \in \p^{\Ht(\p)-1}$ and $f^{p-1} \not\in \m^{[p]}$, where $\m$ is the irrelevant ideal of $R$. Observe that $\cR^s(J)$ is Noetherian since the symbolic powers and the ordinary powers of $J_G$ coincide \cite{Ohtani}. Lemma \ref{lemma CriteriaForStrongFregularity} implies that $\cR^s(J_G)$ is strongly $F$-regular.
\end{proof}

The next one is our final result.

\begin{theorem}
Let $G$ be a closed graph. If $J_G$ is unmixed, then $\cR^s(J_G)$ is strongly $F$-regular.
\end{theorem}

\begin{proof}
Since $G$ is closed, $G$ is traceable \cite[Proposition 1.4]{MR3882405}. Thus, we can relabel the vertices of $G$ in such a way that the vertices $1, 2, \dots, d$ form a path in $G$. Let $f = y_1 f_{1,2} \dots f_{d-1,d} x_d$ and let $g = f/f_{1,2}$. Observe that $g \in J_G^{d-2}$, and so, $g \in \p^{d-2}$ for every minimal prime $\p$ of $J_G$. Observe that $f^{p-1} \not\in \m^{[p]}$, where $\m$ is the irrelevant ideal of $R$. Finally, we prove that $(\cR^s(J_G))_{f_{1,2}}$ is strongly $F$-regular. Note that ${J_G}_{f_{1,2}} = R_{f_{1,2}}$. Thus, $(\cR^s(J_G))_{f_{1,2}} = R_{f_{1,2}}[t]$. Since $R$ is strongly $F$-regular, then $R_{f_{1,2}}$ is strongly $F$-regular. This implies that $R_{f_{1,2}}[t] = (\cR^s(J_G))_{f_{1,2}}$ is strongly $F$-regular. Observe that $\cR^s(J)$ is Noetherian since the symbolic powers and the ordinary powers of $J_G$ coincide \cite{MR4143239}. We conclude from Lemma \ref{lemma CriteriaForStrongFregularity} that $\cR^s(J_G)$ is strongly $F$-regular.
\end{proof}

\section*{Acknowledgements}

I thank Alessandro De Stefani, Claudia Miller, Eloisa Grifo, Jack Jefrries, Jonathan Monta\~no, Jonathan Trevi\~no-Marroqu\'in and Luis N\'u\~nez-Betancourt for helpful comments and discussions.

\bibliographystyle{alpha}
\bibliography{ref.bib}
\end{document}